\documentclass[12pt, two]{article}
\usepackage{}

\usepackage{CJK}
\usepackage{amsthm}
\usepackage{amsmath,amssymb,mathrsfs}
\usepackage{amsfonts}
\usepackage{yhmath}
\usepackage{graphics}
\usepackage{amssymb}
\usepackage{epsfig}
\usepackage{latexsym}
\usepackage{indentfirst}
\usepackage[top=1in, bottom=1in, left=1.25in, right=1.25in]{geometry}
\numberwithin{equation}{section}
\usepackage{CJK}
\allowdisplaybreaks

\newtheorem{theorem}{Theorem}[section]
\newtheorem{lemma}{Lemma}[section]

\newtheorem{remark}{Remark}[section]

\numberwithin{equation}{section} 

\title{
Rigidity theorems of complete K\"ahler-Einstein manifolds and complex space forms
\footnotetext[0]{ This research was supported by NSFC grant No. 11271071 and LMNS, Fudan.}
}

\author{Tian Chong, Yuxin Dong, Hezi Lin and Yibin Ren}
\date{}
\begin{document}


\maketitle
\begin{abstract}
We derive some elliptic differential inequalities from the Weitzenb\"ock
formulas for the traceless Ricci tensor of a K\"ahler manifold with constant
scalar curvature and the Bochner tensor of a K\"ahler-Einstein manifold
respectively. Using elliptic estimates and maximum principle, some $L^p$ and
$L^\infty $ pinching results are established to characterize
K\"ahler-Einstein manifolds among K\"ahler manifolds with constant scalar
curvature, and others are given to characterize complex space forms among
K\"ahler-Einstein manifolds. Finally, these pinching results may be combined
to characterize complex space forms among K\"ahler manifolds with constant
scalar curvature.
\end{abstract}

\section{Introduction}

One of the major problems in geometry is to investigate the rigidity
phenomena of some canonical geometric structures on manifolds. Various
geometric invariants (tensors or quantities) have been introduced to measure
the deviation of a general structure from some canonical one. For a
Riemannian manifold, the traceless Ricci tensor measures its deviation from
an Einstein manifold, while the Weyl curvature tensor measures its deviation
from a conformal flat manifold. These tensors have been used to establish
some rigidity theorems for some special Riemannian manifolds (cf. [HV],
[IS], [Ki], [PRS], [Sh1,2], etc.).

Over the past decades, much effort has been made to establish the existence
of K\"ahler metrics with constant scalar curvature on a compact K\"ahler
manifold (cf. [Ti], [Do], [LS], [Ch] and the references therein). Among these
metrics, K\"ahler-Einstein metrics form a notable subclass, which plays an
important role in both complex geometry and physics. Besides the existence,
the uniqueness and rigidity of these canonical K\"ahler metrics are also
important for geometric applications. Back to early 50's, Calabi had proved
the uniqueness for K\"ahler-Einstein metrics with nonpositive scalar
curvature. In 1986, Bando and Mabuchi [BM] showed the uniqueness for
K\"ahler-Einstein metrics with positive scalar curvature. Actually their
uniqueness results were established within a given K\"ahler class. In [IN],
Itoh and Nakagawa obtained some local rigidity results of a
K\"ahler-Einstein metric in the moduli space of Einstein metrics by means of
the variational stability. On the other hand, complete noncompact
K\"ahler-Einstein manifolds have also received much attention (cf. [TY1],
[TY2], [Ku]).

In this paper, we will consider complete K\"ahler manifolds with constant
scalar curvature and investigate the following two problems:

\textbf{(A)} The rigidity of K\"ahler-Einstein metrics among K\"ahler
metrics with constant scalar curvature;

\textbf{(B)} The rigidity of K\"ahler metrics with constant holomorphic
sectional curvature among K\"ahler-Einstein metrics. $
\begin{array}{ccccc}
&  &  &  &
\end{array}
\begin{array}{ccccc}
&  &  &  &
\end{array}
\begin{array}{ccccc}
&  &  &  &
\end{array}
\begin{array}{cccc}
&  &  &
\end{array}
\begin{array}{ccccc}
&  &  &  &
\end{array}
\begin{array}{cc}
&
\end{array}
$ As in the real case, we may use the traceless Ricci tensor $E$ to measure
the deviation of a K\"ahler metric from a K\"ahler-Einstein metric. In 1949,
Bochner introduced the so-called Bochner curvature tensor $B$ on a K\"ahler
manifold, which is an analogue of the Weyl curvature tensor. It seems a
little tautological to say that the Bochner curvature tensor measures the
deviation of a K\"ahler metric from a Bochner flat K\"ahler metric.
Nevertheless, for a K\"ahler-Einstein metric, the Bochner curvature tensor
measures directly the difference of its curvature tensor from that of the
K\"ahler metric with constant holomorphic sectional curvature (see (4.1) in
\S 4). In order to study Problems A and B, we will derive the Weitzenb\"ock
formulas for $|E|^2$ and $|B|^2$ respectively. First, note that if the
scalar curvature is constant, then $E$ is a $(1,1)$-type Codazzi tensor.
Next, if the metric is K\"ahler-Einstein, then $B$ satisfies the second
Bianchi identidy, which exhibits a Codazzi type property too. These
properties for $E$ and $B$, combined with some refined Kato inequalities,
enable us to deduce some differential inequalities for $|E|$ and $|B|$
respectively. We will treat these differential inequalities on both complete
noncompact and compact K\"ahler manifolds by means of elliptic estimates and
maximum principle. Some $L^p$ and $L^\infty $ pinching results will be
established to characterize K\"ahler-Einstein manifolds among complete
K\"ahler manifolds with constant scalar curvature, and others will be given
to characterize complex space forms among K\"ahler-Einstein manifolds.
Consequently we may also characterize complex space forms among complete
K\"ahler manifolds with constant scalar curvature. Finally, we would like to
mention that the authors [DLR] have established similar results for complete
Sasakian manifolds.
\section{Preliminaries}
Let $(M,g)$ be a smooth Riemannian $n$-manifold with dimension $n\geq 2$ and
let $S_g$ denote the scalar curvature of $g$. The Yamabe constant is defined
by
\begin{equation*}
\Lambda (M,g)=\underset{0\neq u\in C_0^\infty (M)}{\inf }\frac{%
\int_M(|\nabla u|^2+\frac{n-2}{4(n-1)}S_gu^2)dV_g}{(\int_M|u|^{\frac{2n}{n-2}%
}dV_g)^{\frac{n-2}n}}.
\end{equation*}
If $\Lambda (M,g)>0$, then one has the following Sobolev type inequality
\begin{equation}
\Lambda (M,g)(\int_M|u|^{\frac{2n}{n-2}}dV_g)^{\frac{n-2}n}\leq
\int_M(|\nabla u|^2+\frac{n-2}{4(n-1)}S_gu^2)dV_g  \tag{2.1}
\end{equation}
for any $u\in C_0^\infty (M)$. It is known that if $(M,g)$ is compact, then
the sign of $\Lambda (M,g)$ is basically determined by the sign of the
scalar curvature in a conformal class (cf. [He1]). However, there are
complete noncompact Riemannian manifolds with both negative scalar curvature
and positive Yamabe constant (cf. [SY], [He2]).

From now on, we assume that $(M^m,g,J)$ is a K\"ahler manifold with complex
dimension $m\geq 2$ and $\Lambda (M,g)>0$. Let $(z_\alpha )$ be a system of
local complex coordinates on $M$ and let $g_{\alpha \overline{\beta }}$ ($%
1\leq \alpha $, $\beta \leq m$) be the components of the K\"ahler metric in
the coordinates. The inverse matrix of $(g_{\alpha \overline{\beta }})$ is
denoted by $(g^{\alpha \overline{\beta }})$. Let $R_{\alpha \overline{\beta }%
\gamma \overline{\delta }}$ and $R_{\alpha \overline{\beta }}$ denote the
components of the curvature tensor and the Ricci tensor respectively. As
usual, we will use the summation convention on repeating indices. The
complex scalar curvature is defined by
\begin{equation*}
R=g^{\alpha \overline{\beta }}R_{\alpha \overline{\beta }}.
\end{equation*}
Note that $S_g=2R$. In this circumstance the Sobolev inequality (2.1)
becomes
\begin{equation}
\Lambda (M,g)(\int_M|u|^{\frac{2m}{m-1}}dV_g)^{\frac{m-1}m}\leq
\int_M(|\nabla u|^2+\frac{m-1}{2m-1}Ru^2)dV_g  \tag{2.2}
\end{equation}
for any $u\in C_0^\infty (M)$.

In [Bo], S. Bochner introduced the Bocher curvature tensor as follows:
\begin{align}
\nonumber B_{\alpha \overline{\beta }\gamma \overline{\delta }}= &R_{\alpha \overline{
\beta }\gamma \overline{\delta }}+\frac 1{m+2}(R_{\alpha \overline{\beta }
}g_{\gamma \overline{\delta }}+R_{\gamma \overline{\beta }}g_{\alpha
\overline{\delta }}+g_{\alpha \overline{\beta }}R_{\gamma \overline{\delta }
}+g_{\gamma \overline{\beta }}R_{\alpha \overline{\delta }})\\
&-\frac R{(m+1)(m+2)}(g_{\alpha \overline{\beta }}g_{\gamma \overline{\delta }
}+g_{\gamma \overline{\beta }}g_{\alpha \overline{\delta }}) \tag{2.3}
\end{align}
which may be regarded as a complex analogue of the Weyl curvature tensor.
Clearly the Bochner tensor $B$ has the same algebraic symmetries as the
curvature tensor of a K\"ahler metric. These includes
\begin{equation}
\left. B_{\alpha \overline{\beta }\gamma \overline{\delta }}=B_{\gamma
\overline{\beta }\alpha \overline{\delta }}=B_{\alpha \overline{\delta }%
\gamma \overline{\beta }},\quad \overline{B_{\alpha \overline{\beta }\gamma
\overline{\delta }}}=B_{\beta \overline{\alpha }\delta \overline{\gamma }%
}.\right.  \tag{2.4}
\end{equation}
In addition, it has the following metric contraction property:
\begin{equation*}
g^{\alpha \overline{\beta }}B_{\alpha \overline{\beta }\gamma \overline{%
\delta }}=0.  \tag{2.5}
\end{equation*}
The traceless Ricci tensor
$$E=E_{\alpha \overline{\beta }}dz^\alpha \otimes
dz^{\overline{\beta }}+E_{\overline{\alpha }\beta }dz^{\overline{\alpha }%
}\otimes dz^\beta $$
is defined by
\begin{equation*}
E_{\alpha \overline{\beta }}=R_{\alpha \overline{\beta }}-\frac Rmg_{\alpha
\overline{\beta }}.
\end{equation*}
Then the Bochner curvature tensor may also be expressed as
\begin{align}
\nonumber B_{\alpha \overline{\beta }\gamma \overline{\delta }}=&R_{\alpha \overline{
\beta }\gamma \overline{\delta }}+\frac 1{m+2}(E_{\alpha \overline{\beta }
}g_{\gamma \overline{\delta }}+E_{\gamma \overline{\beta }}g_{\alpha
\overline{\delta }}+g_{\alpha \overline{\beta }}E_{\gamma \overline{\delta }
}+g_{\gamma \overline{\beta }}E_{\alpha \overline{\delta }}) \\
&+\frac R{m(m+1)}(g_{\alpha \overline{\beta }}g_{\gamma \overline{\delta }
}+g_{\gamma \overline{\beta }}g_{\alpha \overline{\delta }}).
\tag{2.6}
\end{align}

For a K\"ahler manifold, the second Bianchi identity is reduced to
\begin{equation}
R_{\alpha \overline{\beta }\gamma \overline{\delta },\lambda }=R_{\alpha
\overline{\beta }\lambda \overline{\delta },\gamma }\quad \text{and}\quad
R_{\alpha \overline{\beta }\gamma \overline{\delta },\overline{\lambda }%
}=R_{\alpha \overline{\beta }\gamma \overline{\lambda },\overline{\delta }}.
\tag{2.7}
\end{equation}
By contracting the indices $\alpha $ and $\overline{\beta }$ in (2.7), we
get
\begin{equation}
R_{\gamma \overline{\delta },\lambda }=R_{\lambda \overline{\delta },\gamma
}\quad \text{and\quad }R_{\gamma \overline{\delta },\overline{\lambda }%
}=R_{\gamma \overline{\lambda },\overline{\delta }}.  \tag{2.8}
\end{equation}
A $(1,1)$-$type$ tensor is called Hermitian symmetric if the matrix of its
components is Hermitian symmetric. An Hermitian symmetric $(1,1)$-type
tensor field with the properties (2.8) will be called a $(1,1)$-type
Codazzi tensor. Clearly the Ricci tensor field is a $(1,1)$-type Codazzi
tensor. Thus, if the scalar curvature $R$ is constant, then the traceless
Ricci tensor field $E$ is also a $(1,1)$-type Codazzi tensor.

The usual Ricci identity for commuting covariant derivatives gives
\begin{equation}
E_{\alpha \overline{\beta },\lambda \overline{\mu }}-E_{\alpha \overline{%
\beta },\overline{\mu }\lambda }=E_{\gamma \overline{\beta }}R_{\overline{%
\gamma }\alpha \lambda \overline{\mu }}+E_{\alpha \overline{\gamma }%
}R_{\gamma \overline{\beta }\lambda \overline{\mu }}  \tag{2.9}
\end{equation}
and
\begin{equation}
B_{\alpha \overline{\beta }\gamma \overline{\delta },\lambda \overline{\mu }%
}-B_{\alpha \overline{\beta }\gamma \overline{\delta },\overline{\mu }%
\lambda }=B_{v\overline{\beta }\gamma \overline{\delta }}R_{\overline{v}%
\alpha \lambda \overline{\mu }}+B_{\alpha \overline{v}\gamma \overline{%
\delta }}R_{v\overline{\beta }\lambda \overline{\mu }}+B_{\alpha \overline{%
\beta }v\overline{\delta }}R_{\overline{v}\gamma \lambda \overline{\mu }%
}+B_{\alpha \overline{\beta }\gamma \overline{v}}R_{v\overline{\delta }%
\lambda \overline{\mu }}.  \tag{2.10}
\end{equation}

We will need the following two lemmas. The first one is an algebraic
inequality.

\begin{lemma} ([Ok]) Let $\lambda _\alpha $, $\alpha =1,...,m$, be
real numbers. If $\sum_{\alpha =1}^m\lambda _\alpha =0$, then
\begin{equation*}
|\sum_{\alpha =1}^m\lambda _\alpha ^3|\leq \frac{m-2}{\sqrt{m(m-1)}}%
(\sum_{\alpha =1}^m\lambda _\alpha ^2)^{3/2}.
\end{equation*}
\end{lemma}
The next one is a gap result for solutions of an elliptic differential
inequality.
\begin{lemma} ([PRS]) Let $(M,g)$ be a complete Riemannian $n$%
-manifold on which the following Euclidean-type Sobolev inequality
\begin{equation*}
C(n)\left( \int_M|u|^{\frac{2n}{n-2}}dV_g\right) ^{\frac{n-2}n}\leq
\int_M|\nabla u|^2dV_g  \tag{2.11}
\end{equation*}
holds for every $u\in C_0^\infty (M)$ with a positive constant $C(n)>0$.
Suppose that $\psi \in Lip_{loc}(M)$ is a nonnegative solution of
\begin{equation*}
\psi \triangle \psi +q(x)\psi ^2\geq A|\nabla \psi |^2\quad \text{(weakly)
on }M
\end{equation*}
satisfying
\begin{equation*}
\int_{B_r}|\psi |^{\frac n2}dV_g=o(r^2)\quad \text{as }r\rightarrow +\infty
\end{equation*}
with $A\in R$, $A+\frac n2-1>0$ and $q(x)\in C^0(M)$. If $\psi $ is not
identically zero, then
\begin{equation*}
||q_{+}(x)||_{L^{\frac n2}(M)}\geq \frac{16C(n)(A+n/2-1)}{n^2}.
\end{equation*}
\end{lemma}

\section{ Rigidity of K\"ahler-Einstein manifolds}

In this section, we consider complete K\"ahler manifolds with constant
scalar curvature. Some $L^p$ and $L^\infty $ pinching results will be
established to characterize K\"ahler-Einstein manifolds among complete
K\"ahler manifolds with constant scalar curvature.

Suppose $(M,g,J)$ is a K\"ahler manifold with constant scalar curvature.
First, we intend to derive the Weitzenb\"ock formula for the traceless Ricci
tensor $E$. Note that $E$ is a $(1,1)$-type Codazzi tensor and
 $$|E|^2=2|E_{\alpha \overline{\beta }}|^2.$$ For simplicity, one may choose a
normal complex coordinate system at a given point. Using (2.6), (2.7) and
(2.10), a direct computation gives
\begin{align}
\nonumber \frac 12\triangle |E|^2 =&\nabla _\lambda \nabla _{\overline{\lambda }
}|E_{\alpha \overline{\beta }}|^2+\nabla _{\overline{\lambda }}\nabla
_\lambda |E_{\alpha \overline{\beta }}|^2 \\
\nonumber =&4E_{\alpha \overline{\beta },\lambda }E_{\overline{\alpha }\beta ,\overline{
\lambda }}+2E_{\alpha \overline{\beta }}E_{\overline{\alpha }\beta ,\lambda
\overline{\lambda }}+2E_{\alpha \overline{\beta }}E_{\overline{\alpha }\beta
,\overline{\lambda }\lambda }\\
\nonumber =& 4E_{\alpha \overline{\beta },\lambda }E_{\overline{\alpha }\beta ,\overline{
\lambda }}+4E_{\alpha \overline{\beta }}R_{\beta \overline{\gamma }
}E_{\gamma \overline{\alpha }}+4E_{\alpha \overline{\beta }}E_{\lambda
\overline{\gamma }}R_{\gamma \overline{\alpha }\beta \overline{\lambda }}\\
\nonumber =& 4E_{\alpha \overline{\beta },\lambda }E_{\overline{\alpha }\beta ,\overline{
\lambda }}+\frac{4m}{m+2}tr(E_{\alpha \overline{\beta }})^3+4E_{\alpha
\overline{\beta }}E_{\lambda \overline{\gamma }}B_{\gamma \overline{\alpha }
\beta \overline{\lambda }}+\frac{4R}{m+1}|E_{\alpha \overline{\beta }}|^2\\
\nonumber =& |\nabla _{}E|^2+\frac{4m}{m+2}tr(E_{\alpha \overline{\beta }})^3+4E_{\alpha
\overline{\beta }}E_{\lambda \overline{\gamma }}B_{\gamma \overline{\alpha }
\beta \overline{\lambda }}+\frac{2R}{m+1}|E|^2  \tag{3.1}
\end{align}
where $|\nabla _{}E|^2=4E_{\alpha \overline{\beta },\lambda }E_{\overline{%
\alpha }\beta ,\overline{\lambda }}$.

In [HV], the authors deduced the Kato's inequality for a traceless Codazzi
tensor. Using a similar method, we may derive the following Kato's
inequality for a $(1,1)$-type Codazzi tensor.

\begin{lemma} Let $C$ be a traceless $(1,1)$-type Codazzi tensor field
on $(M^m,g,J)$. Then
\begin{equation}
|\nabla _{}|C||^2\leq \frac m{m+1}|\nabla _{}C|^2  \tag{3.2}
\end{equation}
at any point where $|C|\neq 0$. In addition, the constant on the right hand
side of the inequality is optimal.
\end{lemma}
\begin{proof}
 Clearly the inequality (3.2) is equivalent to
\begin{equation}
\frac 14|\nabla _{}|C|^2|^2\leq \frac m{m+1}|C|^2|\nabla _{}C|^2.  \tag{3.3}
\end{equation}
For any given point $p\in M$, one may choose a system of complex coordinates
$(z_\alpha )$ such that
\begin{equation*}
g_{\alpha \overline{\beta }}=\delta _{\alpha
\overline{\beta }}
\quad \text{and\quad }C_{\alpha \overline{\beta }}=\lambda _\alpha \delta
_{\alpha \overline{\beta }}
\end{equation*}
at this point. Since $(C_{\alpha \overline{%
\beta }})$ is Hermitian symmetric, each eigenvalue $\lambda _\alpha $ is
real. Write
$$C=C_{\alpha \overline{\beta }}dz^\alpha \otimes dz^{\overline{%
\beta }}+C_{\overline{\alpha }\beta }dz^{\overline{\alpha }}\otimes dz^\beta.$$
So
 $$|C|^2=2\ C_{\alpha \overline{\beta }}C^{\alpha \overline{\beta }%
}=2C_{\alpha \overline{\beta }}C_{\overline{\alpha }\beta }.$$
First, we compute
\begin{eqnarray*}
|\nabla _{}|C|^2|^2&=&(|C|^2)_\gamma (|C|^2)_{\overline{\gamma }}+(|C|^2)_{%
\overline{\gamma }}(|C|^2)_\gamma \\
&=& 32C_{\alpha \overline{\beta },\gamma }C_{\overline{\alpha }\beta }C_{\mu
\overline{v},\overline{\gamma }}C_{\overline{\mu }v} \\
&=& 32(\sum_\alpha C_{\alpha \overline{\alpha },\gamma }C_{\overline{\alpha }%
\alpha })(\sum_\mu C_{\mu \overline{\mu },\overline{\gamma }}C_{\overline{%
\mu }\mu }) \\
&=& 32\sum_\gamma |\sum_\alpha C_{\alpha \overline{\alpha },\gamma }C_{%
\overline{\alpha }\alpha }|^2 \\
&\leq& 32\sum_\gamma (\sum_\alpha |C_{\alpha \overline{\alpha }}||C_{\alpha
\overline{\alpha },\gamma }|)^2.
\end{eqnarray*}
Next, since $C$ is a $(1,1)$-type Codazzi tensor, we discover
\begin{align*}
|\nabla _{}C|^2 &=4\sum_{\alpha ,\beta ,\gamma }|C_{\alpha \overline{\beta }%
,\gamma }|^2 \\
&=4\sum_\gamma \{|C_{\gamma \overline{\gamma },\gamma }|^2+\sum_{\alpha \neq
\gamma }|C_{\alpha \overline{\alpha },\gamma }|^2+\sum_{\alpha \neq \gamma
}|C_{\alpha \overline{\gamma },\gamma }|^2\}+\text{positive terms} \\
&\geq 4\sum_\gamma \{|C_{\gamma \overline{\gamma },\gamma }|^2+2\sum_{\alpha
\neq \gamma }|C_{\alpha \overline{\alpha },\gamma }|^2\}.
\end{align*}
In order to prove (3.3), one only needs to verify the following inequality
\begin{equation}
\left. \sum_\gamma (\sum_\alpha |C_{\alpha \overline{\alpha }}||C_{\alpha
\overline{\alpha },\gamma }|)^2\leq \frac m{2(m+1)}|C|^2\sum_\gamma
\{|C_{\gamma \overline{\gamma },\gamma }|^2+2\sum_{\alpha \neq \gamma
}|C_{\alpha \overline{\alpha },\gamma }|^2\}.\right.  \tag{3.4}
\end{equation}
Set $\mu _\alpha =C_{\alpha \overline{\alpha },\gamma }$ for any fixed $%
\gamma $. Note that $\lambda _\alpha =C_{\alpha \overline{\alpha }}$.
Consequently
\begin{equation}
\sum_\alpha \lambda _\alpha =0\text{,\quad }\sum_\alpha \mu _\alpha =0,
\tag{3.5}
\end{equation}
since $C$ is traceless. It follows from (3.5) and the Cauchy-Schwarz
inequality that
\begin{eqnarray*}
|\mu _\gamma |^2 =|\sum_{\alpha \neq \gamma }\mu _\alpha |^2\leq
(\sum_{\alpha \neq \gamma }|\mu _\alpha |)^2
\leq (m-1)(\sum_{\alpha \neq \gamma }|\mu _\alpha |^2).
\end{eqnarray*}
Hence
\begin{align}
\nonumber |\mu _\gamma |^2+2\sum_{\alpha \neq \gamma }|\mu _\alpha |^2
&=|\mu _\gamma|^2+\frac 1m[(m-1)\sum_{\alpha \neq \gamma }|\mu _\alpha
|^2+(m+1)\sum_{\alpha \neq \gamma }|\mu _\alpha |^2] \\
\nonumber &\geq |\mu _\gamma |^2+\frac 1m[|\mu _\gamma |^2+(m+1)\sum_{\alpha \neq
\gamma }|\mu _\alpha |^2] \\
&= \frac{m+1}m\sum_\alpha |\mu _\alpha |^2.  \tag{3.6}
\end{align}
Using (3.6) and the Cauchy-Schwarz inequality, we find
\begin{eqnarray*}
\frac{(\sum_\alpha |\lambda _\alpha |^2)(|\mu _\gamma |^2+2\sum_{\alpha \neq
\gamma }|\mu _\alpha |^2)}{(\sum_\alpha |\lambda _\alpha ||\mu _\alpha |)^2}
&\geq &\frac{m+1}m\frac{(\sum_\alpha |\lambda _\alpha |^2)(\sum_\alpha |\mu
_\alpha |^2)}{(\sum_\alpha |\lambda _\alpha ||\mu _\alpha |)^2} \\
&\geq &\frac{m+1}m
\end{eqnarray*}
which implies immediately the inequality (3.4).
\end{proof}

Another equivalent expression of (2.5) is
\begin{align*}
R_{\alpha \overline{\beta }\gamma \overline{\delta }}=& B_{\alpha \overline{%
\beta }\gamma \overline{\delta }}-\frac 1{m+2}(E_{\alpha \overline{\beta }%
}g_{\gamma \overline{\delta }}+E_{\gamma \overline{\beta }}g_{\alpha
\overline{\delta }}+g_{\alpha \overline{\beta }}E_{\gamma \overline{\delta }%
}+g_{\gamma \overline{\beta }}E_{\alpha \overline{\delta }}) \\
&-\frac R{m(m+1)}(g_{\alpha \overline{\beta }}g_{\gamma \overline{\delta }%
}+g_{\gamma \overline{\beta }}g_{\alpha \overline{\delta }}),
\end{align*}
which tells us that the curvature tensor of a K\"ahler manifold can be
decomposed into three orthogonal parts with respect to the Hermitian
structure. Now we want to estimate the third term on the right hand side of
(3.1) by using the same technique as in [Hu] for treating a similar
contracted term of the traceless Ricci tensor and the Weyl tensor.

\begin{lemma} The inequality
\begin{equation*}
|E_{\alpha \overline{\beta }}E_{\lambda \overline{\gamma }}B_{\gamma
\overline{\alpha }\beta \overline{\lambda }}|\leq \frac 14\sqrt{\frac{%
2m^2+4m+3}{2(m+1)(m+2)}}|B||E|^2
\end{equation*}
holds on any K\"ahler $m$-manifold.
\end{lemma}
\begin{proof}
We define a curvature-like tensor
\begin{align*}
V\ =&\ \ (E_{\alpha \overline{\beta }}E_{\delta \overline{\gamma }}+E_{\alpha
\overline{\gamma }}E_{\delta \overline{\beta }})dz^\alpha \otimes dz^{%
\overline{\beta }}\otimes dz^{\overline{\gamma }}\otimes dz^\delta \\
&+(E_{\beta \overline{\alpha }}E_{\gamma \overline{\delta }}+E_{\gamma
\overline{\alpha }}E_{\beta \overline{\delta }})dz^{\overline{\alpha }%
}\otimes dz^\beta \otimes dz^\gamma \otimes dz^{\overline{\delta }} \\
&-(E_{\beta \overline{\alpha }}E_{\delta \overline{\gamma }}+E_{\delta
\overline{\alpha }}E_{\beta \overline{\gamma }})dz^{\overline{\alpha }%
}\otimes dz^\beta \otimes dz^{\overline{\gamma }}\otimes dz^\delta \\
&-(E_{\alpha \overline{\beta }}E_{\gamma \overline{\delta }}+E_{\alpha
\overline{\delta }}E_{\gamma \overline{\beta }})dz^\alpha \otimes dz^{%
\overline{\beta }}\otimes dz^\gamma \otimes dz^{\overline{\delta }}. \tag{3.7}
\end{align*}
Clearly $V$ has the same symmetries as the curvature tensor of a K\"ahler
manifold. So it can be decomposed into three orthogonal parts with respect
to the Hermitian structure: $V=V_1+V_2+V_3$. Here $V_1$, $V_2$ and $V_3$
correspond to the `Bochner curvature' part, the `traceless Ricci' part and
the `scalar curvature' part of $V$ respectively. To express $V_i$
explicitly, let's introduce
\begin{equation*}
V_{\alpha \overline{\beta }}^E=V_{\alpha \overline{\beta }}^{Ric}-\frac
Kmg_{\alpha \overline{\beta }}
\end{equation*}
where $$V_{\alpha \overline{\beta }}^{Ric}=g^{\gamma \overline{\delta }%
}V_{\alpha \overline{\beta }\overline{\delta }\gamma }=g^{\gamma \overline{%
\delta }}E_{\alpha \overline{\delta }}E_{\gamma \overline{\beta }}$$
and
$$K=g^{\alpha \overline{\beta }}V_{\alpha \overline{\beta }}^{Ric}=\frac
12|E|^2.$$
Therefore the components of $V_2$ and $V_3$ are given by
\begin{equation*}
(V_2)_{\alpha \overline{\beta }\gamma \overline{\delta }}=-\frac
1{m+2}(V_{\alpha \overline{\beta }}^Eg_{\gamma \overline{\delta }}+V_{\gamma
\overline{\beta }}^Eg_{\alpha \overline{\delta }}+g_{\alpha \overline{\beta }%
}V_{\gamma \overline{\delta }}^E+g_{\gamma \overline{\beta }}V_{\alpha
\overline{\delta }}^E)
\end{equation*}
and
\begin{equation*}
(V_3)_{\alpha \overline{\beta }\gamma \overline{\delta }}=-\frac
K{m(m+1)}(g_{\alpha \overline{\beta }}g_{\gamma \overline{\delta }%
}+g_{\gamma \overline{\beta }}g_{\alpha \overline{\delta }}).
\end{equation*}

As before, we may assume $g_{\alpha \overline{\beta }}=\delta _{\alpha
\overline{\beta }}$ at a given point. From (2.4) and (3.7), we have
\begin{equation}
\left. 8E_{\alpha \overline{\beta }}E_{\lambda \overline{\gamma }}B_{\gamma
\overline{\alpha }\beta \overline{\lambda }}=\langle B,V\rangle =\langle
B,V_1\rangle \right.  \tag{3.8}
\end{equation}
where $\langle \cdot ,\cdot \rangle $ denotes the Hermitian inner product
induced from $g$. Set
 $$Z=E_{\alpha \overline{\beta }}E_{\beta \overline{%
\gamma }}E_{\gamma \overline{\delta }}E_{\delta \overline{\alpha }}.$$
 A direct calculation yields that
\begin{align*}
\frac 14|V|^2&=(E_{\alpha \overline{\beta }}E_{\gamma \overline{\delta }%
}+E_{\alpha \overline{\delta }}E_{\gamma \overline{\beta }})(E_{\overline{%
\alpha }\beta }E_{\overline{\gamma }\delta }+E_{\overline{\alpha }\delta }E_{%
\overline{\gamma }\beta }) \\
&=\frac 12|E|^4+2Z, \tag{3.9}\\
&{}\\
\frac 14|V_2|^2&=\frac 4{m+2}V_{\alpha \overline{\beta }}^EV_{\overline{%
\alpha }\beta }^E \\
&=\frac 4{m+2}(E_{\alpha \overline{\lambda }}E_{\lambda \overline{\beta }%
}-\frac 1{2m}|E|^2\delta _{\alpha \overline{\beta }})(E_{\overline{\alpha }%
\mu }E_{\overline{\mu }\beta }-\frac 1{2m}|E|^2\delta _{\overline{\alpha }%
\beta }) \\
& =\frac 4{m+2}Z-\frac 1{m(m+2)}|E|^4,  \tag{3.10}
\end{align*}
and
\begin{align*}
\frac 14|V_3|^2=\frac 1{2m(m+1)}|E|^4.  \tag{3.11}
\end{align*}
From (3.9), (3.10) and (3.11), we deduce
\begin{align}
\nonumber |V_1|^2&=|V|^2-|V_2|^2-|V_3|^2 \\
\nonumber &=\frac{8m}{m+2}Z+\frac{2(m+1)(m+2)+2}{(m+1)(m+2)}|E|^4 \\
\nonumber &\leq \frac{2m}{m+2}|E|^4+\frac{2(m+1)(m+2)+2}{(m+1)(m+2)}|E|^4 \\
&=\frac{4m^2+8m+6}{(m+1)(m+2)}|E|^4.\tag{3.12}
\end{align}
It follows from (3.8) and (3.12) that
\begin{eqnarray*}
|E_{\alpha \overline{\beta }}E_{\lambda \overline{\gamma }}B_{\gamma
\overline{\alpha }\beta \overline{\lambda }}| &\leq &\frac 18|\langle
B,V_1\rangle | \\
&\leq &\frac 18|B||V_1| \\
&\leq &\frac 14\sqrt{\frac{2m^2+4m+3}{2(m+1)(m+2)}}|B||E|^2.
\end{eqnarray*}
\end{proof}

Using Lemmas 2.1 and 3.2, we get from (3.1) that
\begin{equation}
\frac 12\triangle |E|^2\geq |\nabla _{}E|^2-\frac{(m-2)\sqrt{2m}}{(m+2)\sqrt{%
(m-1)}}|E|^3-\sqrt{\frac{2m^2+4m+3}{2(m+1)(m+2)}}|B||E|^2
+\frac{2R}{m+1}|E|^2 \tag{3.13}
\end{equation}
and thus using Lemma 3.1, we find
\begin{align*}
|E|\triangle |E|\geq& |\nabla _{}E|^2-|\nabla _{}|E||^2-\frac{(m-2)\sqrt{2m}}{%
(m+2)\sqrt{(m-1)}}|E|^3
-\sqrt{\frac{2m^2+4m+3}{2(m+1)(m+2)}}|B||E|^2\\
&+\frac{2R}{m+1}|E|^2 \\
\geq& \frac 1m|\nabla _{}|E||^2+\frac{2R}{m+1}|E|^2-\frac{(m-2)\sqrt{2m}}{%
(m+2)\sqrt{(m-1)}}|E|^3\\
&-\sqrt{\frac{2m^2+4m+3}{2(m+1)(m+2)}}|B||E|^2.\tag{3.14}
\end{align*}

\begin{theorem}
 Let $(M,g,J)$ be a complete noncompact K\"ahler $m$%
-manifold ($m\geq 2$) with zero scalar curvature and positive Yamabe
constant $\Lambda (M,g)$. Assume that
\begin{equation}
\sqrt{2}||E||_{L^m(M)}+||B||_{L^m(M)}<\frac{4\Lambda (M,g)(m^2-m+1)}{m^3}%
\sqrt{\frac{2(m+1)(m+2)}{2m^2+4m+3}}.  \tag{3.15}
\end{equation}
Then $M$ is a Ricci-flat K\"ahler manifold.
\end{theorem}
\begin{proof}
 Since $R=0$, the differential inequality (3.14) becomes
\begin{equation}
|E|\triangle |E|+\left( \frac{m-2}{m+2}\sqrt{\frac{2m}{m-1}}|E|+\sqrt{\frac{%
2m^2+4m+3}{2(m+1)(m+2)}}|B|\right) |E|^2\geq \frac 1m|\nabla _{}|E||^2.
\tag{3.16}
\end{equation}
Under the assumptions that $\Lambda (M,g)>0$ and $R=0$, the following
Euclidean-type Sobolev inequality
\begin{equation*}
\Lambda (M,g)\left( \int_M|u|^{\frac{2m}{m-1}}dV_g\right) ^{\frac{m-1}m}\leq
\int_M|\nabla _{}u|^2dV_g
\end{equation*}
holds for any $u\in C_0^\infty (M)$.

We have to show that $|E|=0$. Clearly (3.15) implies that $\int_M|E|^mdV_g$
is finite, and thus
\begin{equation*}
\int_{B_r}|E|^mdV_g=o(r^2)\ \  as\ \  r\rightarrow \infty.
\end{equation*}
If $|E|$ is not identically zero, applying Lemma 2.2 to (3.16), we get
\begin{equation*}
||\frac{m-2}{m+2}\sqrt{\frac{2m}{m-1}}|E|+\sqrt{\frac{2m^2+4m+3}{2(m+1)(m+2)}%
}|B|||_{L^m(M)}\geq \frac{4\Lambda (M,g)(m^2-m+1)}{m^3}.
\end{equation*}
Note that $\sqrt{\frac{2m^2+4m+3}{2(m+1)(m+2)}}>\frac{m-2}{m+2}\sqrt{\frac
m{m-1}}$ for $m\geq 2$. Consequently
\begin{equation*}
\sqrt{\frac{2m^2+4m+3}{2(m+1)(m+2)}}\left( \sqrt{2}%
||E||_{L^m(M)}+||B||_{L^m(M)}\right) \geq \frac{4\Lambda (M,g)(m^2-m+1)}{m^3}
\end{equation*}
which contradicts to (3.15). Hence we conclude that $E=0$, that is, $(M,g,J)$
is K\"ahler-Einstein.
\end{proof}

Next we deal with the case that $R<0$. Although in this case, the Sobolev
inequality (2.2) implies the Euclidean-type Sobolev inequality (2.10) with $%
C(n)=\Lambda (M,g)$ and $n=2m$, the direct application of Lemma 2.2 to
(3.16) does not yield a nice gap result as in Theorem 3.1. Inspired by a
technique in [Ki], we establish the following result.
\begin{theorem}
 Let $(M,g,J)$ be a complete noncompact K\"ahler $m$%
-manifold ($m\geq 3$) with constant negative scalar curvature $R$ and
positive Yamabe constant $\Lambda (M,g)$. Suppose that
\begin{equation*}
\int_{B_r}|E|^2dV_g=o(r^2)\ \ as \ \ r\rightarrow \infty
\end{equation*}
 where $B_r$ denotes a geodesic ball of radius $r$
relative to some fixed point $x_0\in M$. If
\begin{equation}
\sqrt{2}||E||_{L^m(M)}+||B||_{L^m(M)}<\frac{\Lambda (M,g)(m+1)}m\sqrt{\frac{%
2(m+1)(m+2)}{2m^2+4m+3}},  \tag{3.17}
\end{equation}
then $(M,g,J)$ is K\"ahler-Einstein.
\end{theorem}
\begin{proof}
 Set $u=|E|$. For any test function $0\leq \phi \in C_0^\infty (M)$,
we get from (3.14) that
\begin{align*}
\int_Mu(\triangle u)\phi ^2dV_g\geq& \int_M\{\frac 1m|\nabla u|^2\phi ^2+%
\frac{2R}{m+1}u^2\phi ^2-\frac{m-2}{m+2}\sqrt{\frac{2m}{m-1}}u^3\phi ^2 \\
&-\sqrt{\frac{2m^2+4m+3}{2(m+1)(m+2)}}|B|u^2\phi ^2\}dV_g. \tag{3.18}
\end{align*}
Using integration by parts and the Schwarz inequality, we deduce
\begin{align*}
\int_Mu(\triangle u)\phi ^2dV &=-\int_M|\nabla u|^2\phi ^2dV_g-2\int_M\phi
u<\nabla u,\nabla \phi >dV_g \\
&\leq (\varepsilon _1-1)\int_M|\nabla u|^2\phi ^2dV_g+\varepsilon
_1^{-1}\int_M|\nabla \phi |^2u^2dV_g  \tag{3.19}
\end{align*}
for any $\varepsilon _1>0$. It follows from (3.18) and (3.19) that
\begin{align*}
(1+\frac 1m-\varepsilon _1)\int_M|\nabla u|^2\phi ^2dV_g
\leq& \int_M\{\varepsilon _1^{-1}|\nabla \phi |^2u^2+\phi ^2[\frac{m-2}{m+2}%
\sqrt{\frac{2m}{m-1}}u^3\\
&+\sqrt{\frac{2m^2+4m+3}{2(m+1)(m+2)}}|B|u^2-\frac{2R%
}{m+1}u^2]\}dV_g.\tag{3.20}
\end{align*}
From the Sobolev inequality (2.2) and the Schwarz inequality, we find
\begin{align*}
\Lambda (M,g)\left( \int_M(\phi u)^{\frac{2m}{m-1}}dV_g\right) ^{\frac{m-1}%
m}\leq& \int_M\{(1+\varepsilon _2)|\nabla u|^2\phi ^2+(1+\varepsilon
_2^{-1})|\nabla \phi |^2u^2 \\
&+\frac{m-1}{2m-1}R(\phi u)^2\}dV_g \tag{3.21}
\end{align*}
for any $\varepsilon _2>0$. Then (3.20) and (3.21) imply
\begin{align*}
\Lambda (M,g)\left( \int_M(\phi u)^{\frac{2m}{m-1}}dV_g\right) ^{\frac{m-1}m}
\leq& \int_M\{A_1|\nabla \phi |^2u^2+A_2R\phi ^2u^2\\
& +A_3|B|\phi ^2u^2+A_4\phi^2u^3\}dV_g \tag{3.22}
\end{align*}
where
\begin{align*}
A_1=& \frac{1+\varepsilon _2}{(1+m^{-1}-\varepsilon _1)\varepsilon _1}%
+1+\varepsilon _2^{-1}, \\
A_2=&\frac{m-1}{2m-1}-\frac{2(1+\varepsilon _2)}{(m+1)(1+m^{-1}-\varepsilon
_1)}, \\
A_3=& \frac{1+\varepsilon _2}{1+m^{-1}-\varepsilon _1}\sqrt{\frac{2m^2+4m+3}{%
2(m+1)(m+2)}}, \\
A_4=& \frac{(1+\varepsilon _2)(m-2)}{(1+m^{-1}-\varepsilon _1)(m+2)}\sqrt{%
\frac{2m}{m-1}}.
\end{align*}
Note that $A_2>0$ for $m\geq 3$ and sufficiently small $\varepsilon _1$ and $%
\varepsilon _2$. Under the assumption (3.17), we may choose sufficiently
small $\varepsilon _1$ and $\varepsilon _2$ such that
\begin{equation}
\sqrt{2}||E||_{L^m(M)}+||B||_{L^m(M)}<\frac{\Lambda
(M,g)(1+m^{-1}-2\varepsilon _1)}{1+2\varepsilon _2}\sqrt{\frac{2(m+1)(m+2)}{%
2m^2+4m+3}}.  \tag{3.23}
\end{equation}
Moreover, the sufficiently small $\varepsilon _1$ and $\varepsilon _2$ also
ensure
\begin{equation}
\{\frac{m-1}{2m-1}-\frac{2(1+\varepsilon _2)}{(m+1)(1+m^{-1}-\varepsilon _1)}%
\}R\phi ^2u^2\leq 0.  \tag{3.24}
\end{equation}
Since $\sqrt{2}A_3\geq A_4$ for $m\geq 3$, we get from (3.22) and (3.24)
that
\begin{equation}
\Lambda (M,g)\left( \int_M(\phi u)^{\frac{2m}{m-1}}dV_g\right) ^{\frac{m-1}%
m}\leq \int_M\{A_1|\nabla \phi |^2u^2+A_3[|B|+\sqrt{2}u]\phi ^2u^2\}dV_g.
\tag{3.25}
\end{equation}
The H\"older inequality gives
\begin{align*}
\int_M|B|\phi ^2u^2dV_g\leq& \left( \int_M|B|^mdV_g\right) ^{\frac 1m}\left(
\int_M(\phi u)^{\frac{2m}{m-1}}dV_g\right) ^{\frac{m-1}m} \\
\int_Mu^3\phi ^2dV_g\leq& \left( \int_M|u|^mdV_g\right) ^{\frac 1m}\left(
\int_M(\phi u)^{\frac{2m}{m-1}}dV_g\right) ^{\frac{m-1}m}. \tag{3.26}
\end{align*}
Hence we may combine (3.23), (3.25) and (3.26) to find
\begin{align*}
&\Lambda (M,g)\left( \int_M(\phi u)^{\frac{2m}{m-1}}dV_g\right) ^{\frac{m-1}m}\\
\leq& A_1\int_M|\nabla \phi |^2u^2dV_g+\frac{\Lambda (M,g)(1+\varepsilon
_2)(1+m^{-1}-2\varepsilon _1)}{(1+2\varepsilon _2)((1+m^{-1}-\varepsilon _1)}%
\left( \int_M(\phi u)^{\frac{2m}{m-1}}dV_g\right) ^{\frac{m-1}m}.
\end{align*}
Consequently
\begin{equation}
\Lambda (M,g)[1-\frac{(1+\varepsilon _2)(1+m^{-1}-2\varepsilon _1)}{%
(1+2\varepsilon _2)((1+m^{-1}-\varepsilon _1)}]\left( \int_M(\phi u)^{\frac{%
2m}{m-1}}dV_g\right) ^{\frac{m-1}m}\leq A_1\int_M|\nabla \phi |^2u^2dV_g.
\tag{3.27}
\end{equation}
Now we let $\phi =\phi _r$ be a family of cut-off functions satisfying
\begin{equation*}
\phi _r\equiv 1\text{ on }B_r;\quad \phi _r\equiv 0\text{ off }B_{2r};\quad
|\nabla \phi _r|\leq \frac 2r\text{ on }B_{2r}-B_r.
\end{equation*}
Then (3.27) becomes
\begin{equation}
\Lambda (M,g)[1-\frac{(1+\varepsilon _2)(1+m^{-1}-2\varepsilon _1)}{%
(1+2\varepsilon _2)((1+m^{-1}-\varepsilon _1)}]\left( \int_{B_r}u^{\frac{2m}{%
m-1}}dV_g\right) ^{\frac{m-1}m}\leq \frac{4A_1}{r^2}\int_{B_{2r}}u^2dV_g.
\tag{3.28}
\end{equation}
Letting $r\rightarrow \infty $ in (3.28), we get
\begin{equation*}
\int_Mu^{\frac{2m}{m-1}}dV_g=0,
\end{equation*}
that is, $u\equiv 0$. Hence $(M,g,J)$ is K\"ahler-Einstein.
\end{proof}

Now we consider the case that $R>0$. The Bonnet-Myers theorem in Riemannian
geometry implies that any complete Einstein manifold with positive scalar
curvature must be compact. Following in this section are two rigidity
results about compact K\"ahler manifolds with positive scalar curvature. Recall that
if $(M,g)$ is compact, the positivity of the scalar curvature guarantees the
positivity of the Yamabe constant $\Lambda (M,g)$.
\begin{theorem}
Let $(M,g,J)$ be a compact K\"ahler $m$-manifold ($%
m\geq 2$) with constant positive scalar curvature $R$. If
\begin{equation}
\sqrt{2}||E||_{L^m(M)}+||B||_{L^m(M)}<\Lambda (M,g)P(m)\sqrt{\frac{%
2(m+1)(m+2)}{2m^2+4m+3}}  \tag{3.29}
\end{equation}
where $P(2)=3/2$ and $P(m)=\frac{2(2m-1)}{m^2-1}$ for $m\geq 3$, then $%
(M,g,J)$ is K\"ahler-Einstein.
\end{theorem}
\begin{proof}
By integrating (3.14) and using the H\"older inequality, we have
\begin{align*}
\int_M|\nabla |E||^2dV_g \leq& \frac 1{1+m^{-1}}||q||_{L^m(M)}\left( \int_M|E|^{\frac{2m}{m-1}%
}dV_g\right) ^{\frac{m-1}m}\\
&-\frac{2R}{(m+1)(1+m^{-1})}\int_M|E|^2dV_g \tag{3.30}
\end{align*}
where
\begin{equation*}
q(x)=\frac{m-2}{m+2}\sqrt{\frac{2m}{m-1}}|E|+\sqrt{\frac{2m^2+4m+3}{
2(m+1)(m+2)}}|B|.
\end{equation*}
Since the condition $R>0$ implies $\Lambda (M,g)>0$, (2.2) gives
\begin{equation}
\Lambda (M,g)\left( \int_M|E|^{\frac{2m}{m-1}}dV_g\right) ^{\frac{m-1}m}\leq
\int_M|\nabla |E||^2dV_g+\frac{(m-1)R}{2m-1}\int_M|E|^2dV_g.  \tag{3.31}
\end{equation}
Substituting (3.31) into (3.30) leads to
\begin{align*}
\int_M|\nabla |E||^2dV_g\leq& \frac 1{\Lambda
(M,g)(1+m^{-1})}||q||_{L^m(M)}\int_M[|\nabla |E||^2+\frac{(m-1)R}{2m-1}%
|E|^2]dV_g \\
&-\frac{2R}{(m+1)(1+m^{-1})}\int_M|E|^2dV_g.
\end{align*}
Consequently
\begin{align*}
\left( 1-\frac 1{\Lambda (M,g)(1+m^{-1})}||q||_{L^m(M)}\right)
\int_M|\nabla |E||^2dV_g+& \\
\left( \frac 2{(m+1)(1+m^{-1})}-\frac{(m-1)}{\Lambda (M,g)(1+m^{-1})(2m-1)}%
||q||_{L^m(M)}\right) R\int_M|E|^2dV_g & \leq 0. \tag{3.32}
\end{align*}
Note again that
\begin{equation*}
||q||_{L^m(M)}\leq \sqrt{\frac{2m^2+4m+3}{2(m+1)(m+2)}}\left( \sqrt{2}%
|E|_{L^m(M)}+||B||_{L^m(M)}\right)
\end{equation*}
for $m\geq 2$. Clearly the condition (3.29) guarantees that the two terms on
the left hand side of the inequality (3.32) are nonnegative. Thus we may
conclude that $E=0$.
\end{proof}
\begin{theorem}
Let $(M,g,J)$ be a compact K\"ahler $m$-manifold ($%
m\geq 2$) with constant positive scalar curvature $R$. If
\begin{equation*}
\sqrt{2}|E|+|B|\leq \frac 2{m+1}\sqrt{\frac{2(m+1)(m+2)}{2m^2+4m+3}}R,
\end{equation*}
then $(M,g,J)$ is K\"ahler-Einstein.
\end{theorem}
\begin{proof}
 From (3.13), we have
\begin{align*}
\frac 12\triangle |E|^2\geq& \{\frac{2R}{m+1}-\frac{m-2}{m+2}\sqrt{\frac{2m}{%
m-1}}|E|-\sqrt{\frac{2m^2+4m+3}{2(m+1)(m+2)}}|B|\}|E|^2 \\
=&\{\frac{2R}{m+1}-\sqrt{\frac{2m^2+4m+3}{2(m+1)(m+2)}}(\sqrt{2}%
|E|+|B|)\}|E|^2\\
&+\sqrt 2 \left( \sqrt{\frac{2m^2+4m+3}{2(m+1)(m+2)}}-\frac{m-2}{m+2}%
\sqrt{\frac {m}{m-1}}\right) |E|^3. \tag{3.33}
\end{align*}
Since $\sqrt{\frac{2m^2+4m+3}{2(m+1)(m+2)}}>\frac{m-2}{m+2}\sqrt{\frac {m}{m-1}%
}$, the integration of (3.33) leads to $E=0$, that is, $(M,g,J)$ is
K\"ahler-Einstein.
\end{proof}

\section{ Rigidity of complex space forms}

In this section, we establish some rigidity results characterizing complex
space forms among complete K\"ahler-Einstein manifolds and complete K\"ahler
manifolds with constant scalar curvature respectively.

Suppose $(M,g,J)$ is a K\"ahler-Einstein manifold of dimension $m$ ($m\geq 2$%
). The Einsteinian condition implies directly that the scalar curvature $R$
is constant and (2.6) becomes
\begin{equation}
B_{\alpha \overline{\beta }\gamma \overline{\delta }}=R_{\alpha \overline{%
\beta }\gamma \overline{\delta }}+\frac R{m(m+1)}(g_{\alpha \overline{\beta }%
}g_{\gamma \overline{\delta }}+g_{\gamma \overline{\beta }}g_{\alpha
\overline{\delta }}).  \tag{4.1}
\end{equation}
In this circumstance the Bochner tensor measures the deviation of a
K\"ahler-Einstein metric from the metric with constant holomorphic
sectional curvature.

We want next to derive the Weitzenb\"ock formula for the Bochner tensor $B$.
Note that $B$ is regarded as a real tensor in $\Lambda ^{1,1}(M)\otimes
\Lambda ^{1,1}(M)$. As in \S 3, we take a normal complex coordinate system
at a given point. So
$$|B|^2=4|B_{\alpha \overline{\beta }\gamma \overline{%
\delta }}|^2.$$
 Using (2.4), (2.7), (2.10) and (4.1), a direct computation
gives
\begin{align*}
\frac 12\triangle |B|^2 &=2\left( \nabla _\lambda \nabla _{\overline{\lambda }%
}|B_{\alpha \overline{\beta }\gamma \overline{\delta }}|^2+\nabla _{%
\overline{\lambda }}\nabla _\lambda |B_{\alpha \overline{\beta }\gamma
\overline{\delta }}|^2\right) \\
&= 8B_{\alpha \overline{\beta }\gamma \overline{\delta },\lambda }B_{\overline{%
\alpha }\beta \overline{\gamma }\delta ,\overline{\lambda }}+4B_{\alpha
\overline{\beta }\gamma \overline{\delta },\lambda \overline{\lambda }}B_{%
\overline{\alpha }\beta \overline{\gamma }\delta }+4B_{\alpha \overline{%
\beta }\gamma \overline{\delta },\overline{\lambda }\lambda }B_{\overline{%
\alpha }\beta \overline{\gamma }\delta } \\
&= 8B_{\alpha \overline{\beta }\gamma \overline{\delta },\lambda }B_{\overline{%
\alpha }\beta \overline{\gamma }\delta ,\overline{\lambda }}+4B_{\alpha
\overline{\beta }\lambda \overline{\delta },\gamma \overline{\lambda }}B_{%
\overline{\alpha }\beta \overline{\gamma }\delta }+4B_{\alpha \overline{%
\beta }\gamma \overline{\lambda },\overline{\delta }\lambda }B_{\overline{%
\alpha }\beta \overline{\gamma }\delta } \tag{4.2}\\
&= 8|B_{\alpha \overline{\beta }\gamma \overline{\delta },\lambda }|^2+8B_{%
\overline{\beta }v\lambda \overline{\delta }}B_{\overline{v}\alpha \gamma
\overline{\lambda }}B_{\overline{\alpha }\beta \delta \overline{\gamma }%
}-16B_{\alpha \overline{\delta }\overline{v}\lambda }B_{v\overline{\lambda }%
\overline{\beta }\gamma }B_{\beta \overline{\gamma }\overline{\alpha }\delta
}+\frac{8R}m|B_{\alpha \overline{\beta }\gamma \overline{\delta }}|^2.
\end{align*}
Let us introduce two $m^2\times m^2$ Hermitian matrices $H$ and $K$ as
follows:
\begin{align*}
H=&(H_{ab\overline{c}\overline{d}})=(B_{\overline{c}ab\overline{d}}), \\
K=&(K_{a\overline{b}\overline{c\overline{d}}})=(B_{a\overline{b}\overline{c}%
d}).
\end{align*}
Then (4.2) becomes
\begin{equation}
\frac 12\triangle |B|^2=|\nabla _{}B|^2+8tr(H^3)-16tr(K^3)+\frac{2R}m|B|^2
\tag{4.3}
\end{equation}
where $|\nabla _{}B|^2=8|B_{\alpha \overline{\beta }\gamma \overline{\delta }%
,\lambda }|^2$. In view of (2.4) and (2.5), we see
\begin{align*}
tr(H)=&tr(K)=0, \\
tr(H^2)=&tr(K^2)=\frac 14|B|^2.
\end{align*}
Consequently Lemma 2.1 yields
\begin{equation}
|tr(H^3)|\leq \frac{m^2-2}{8\sqrt{m^2(m^2-1)}}|B|^3, \tag{4.4}
\end{equation}
and
\begin{equation}
|tr(K^3)|\leq \frac{m^2-2}{8\sqrt{m^2(m^2-1)}}|B|^3. \tag{4.5}
\end{equation}
From (4.3), (4.4) and (4.5), we deduce
\begin{equation}
\frac 12\triangle |B|^2\geq |\nabla _{}B|^2-\frac{3(m^2-2)}{\sqrt{m^2(m^2-1)}%
}|B|^3+\frac{2R}m|B|^2.  \tag{4.6}
\end{equation}
In order to estimate the first term on the right hand side of (4.6), we need
the following

\begin{lemma}([BKN]) Let $T_1$ and $T_2$ be tensors having the same
symmetries as the curvature tensor and the covariant derivative of the
curvature tensor of an Einstein metric on $n$-manifold respectively. Then
there exists $\delta (n)$ such that
\begin{equation*}
(1+\delta (n))|\langle T_1,T_2\rangle |^2\leq |T_1|^2|T_2|^2\text{, }
\end{equation*}
where $\langle T_1,T_2\rangle $ is a $1$-form defined by $\langle
T_1,T_2\rangle (X)=\langle T_1,T_2(X)\rangle $ for a tangent $X$. Moreover,
if $g$ is K\"ahler, we can take $\delta (n)=\frac 4{n+2}=\frac 2{m+1}$,
where $n=2m$.
\end{lemma}

By applying Lemma 4.1 to $T_1=B$ and $T_2=\nabla _{}B$, we find
\begin{equation*}
\frac 14|\nabla _{}|B|^2|^2=|\langle B,\nabla _{}B\rangle |^2
\leq \frac{m+1}{m+3}|B|^2|\nabla _{}B|^2.
\end{equation*}
Note also that $|\nabla _{}|B|^2|^2=4|B|^2|\nabla _{}|B||^2$. Consequently
\begin{equation}
|\nabla _{}B|^2\geq \frac{m+3}{m+1}|\nabla _{}|B||^2  \tag{4.7}.
\end{equation}
Hence (4.6) and (4,7) imply
\begin{equation}
|B|\triangle |B|\geq \frac 2{m+1}|\nabla _{}|B||^2-\frac{3(m^2-2)}{\sqrt{%
m^2(m^2-1)}}|B|^3+\frac{2R}m|B|^2.  \tag{4.8}
\end{equation}

First, we consider the case that $R=0$. As a result of Lemma 2.2, we have

\begin{theorem}
Let $(M,g,J)$ be a complete noncompact
K\"ahler-Einstein $m$- manifold ($m\geq 2$) with $R=0$ and $\Lambda (M,g)>0$%
. If
\begin{equation}
||B||_{L^m(M)}<\frac{4\Lambda (M,g)(m^2+1)\sqrt{m^2-1}}{3m(m+1)(m^2-2)},
\tag{4.9}
\end{equation}
then $(M,g,J)$ is of constant holomorphic sectional curvature $0$.
Furthermore, if $M$ is simply connected, then $(M,g,J)$ is biholomorphically
isometric to the complex Euclidean space $C^m$.
\end{theorem}
\begin{proof}
Since $R=0$, the Sobolev inequality (2.2) provides an Euclidean-type
Sobolev inequality with Sobolev constant $C(2m)=\Lambda (M,g)$, and (4.8)
becomes
\begin{equation}
|B|\triangle |B|+\left( \frac{3(m^2-2)}{\sqrt{m^2(m^2-1)}}|B|\right)
|B|^2\geq \frac 2{m+1}|\nabla _{}|B||^2.  \tag{4.10}
\end{equation}
The assumption (4.9) implies that
\begin{equation*}
\int_{B_r}|B|^mdV_g=o(r^2)\text{\quad as }r\rightarrow \infty .
\end{equation*}
Thus, if $|B|$ is not identically zero, we get from Lemma 2.2 and (4.10)
that
\begin{equation*}
||B||_{L^m(M)}\geq \frac{4\Lambda (M,g)(m^2+1)\sqrt{m^2-1}}{3m(m+1)(m^2-2)}
\end{equation*}
which contradicts to (4.9). Thus $B=0$ and therefore (4.1) yields that $%
R_{\alpha \overline{\beta }\gamma \overline{\delta }}=0$. This shows that $%
(M,g,J)$ is of constant holomorphic sectional curvature $0$. Consequently,
if $M$ is simply connected, then $(M,g,J)$ is biholomorphically isometric to
$C^m$ (cf. Theorem 7.9 in [KN]).
\end{proof}

Next we present the following rigidity result for the case $R<0$. Since its
proof goes almost the same way as that for Theorem 3.2, we will describe the
argument briefly .

\begin{theorem}
 Let $(M,g,J)$ be a complete noncompact
K\"ahler-Einstein $m$-manifold ($m\geq 4$) with $R<0$ and $\Lambda (M,g)>0$.
Suppose $\int_{B_r}|B|^2dV_g=o(r^2)$ as $r\rightarrow \infty $. If
\begin{equation}
||B||_{L^m(M)}<\frac{\Lambda (M,g)(m+3)\sqrt{m^2(m^2-1)}}{3(m+1)(m^2-2)},
\tag{4.11}
\end{equation}
then $(M,g,J)$ is of constant holomorphic sectional curvature $\frac{2R}{%
m(m+1)}$.
\end{theorem}
\begin{proof}
Set $v=|B|$. For any test function $\phi \in C_0^\infty (M)$, it
follows from (4.8) that
\begin{equation}
\int_Mv(\triangle v)\phi ^2dV_g\geq \int_M\{\frac 2{m+1}|\nabla _{}v|^2\phi
^2+\frac{2R}mv^2\phi ^2-\frac{3(m^2-2)}{\sqrt{m^2(m^2-1)}}v^3\phi ^2\}dV_g.
\tag{4.12}
\end{equation}
As we derive (3.22) from (3.18), the same process allows us to get from
(4.12) the following inequality
\begin{equation*}
\left. \Lambda (M,g)\left( \int_M(\phi v)^{\frac{2m}{m-1}}dV_g\right) ^{%
\frac{m-1}m}\leq \int_M\{B_1|\nabla _{}\phi |^2v^2+B_2R\phi ^2v^2+B_3\phi
^2v^3\}dV_g,\right.
\end{equation*}
where
\begin{align*}
B_1=& \frac{1+\varepsilon _2}{\varepsilon _1[1+2(m+1)^{-1}-\varepsilon _1]}%
+1+\varepsilon _2^{-1}, \\
B_2=& \frac{m-1}{2m-1}-\frac{2(1+\varepsilon _2)}{m[1+2(m+1)^{-1}-\varepsilon
_1]}, \\
B_3=& \frac{3(1+\varepsilon _2)(m^2-2)}{[1+2(m+1)^{-1}-\varepsilon _1]\sqrt{%
m^2(m^2-1)}}.
\end{align*}
Note that $B_2>0$ for $m\geq 4$ and sufficiently small $\varepsilon _1$ and $%
\varepsilon _2$. Since $R<0$, we use the H\"older inequality to find
\begin{align*}
\Lambda (M,g)\left( \int_M(\phi v)^{\frac{2m}{m-1}}dV_g\right) ^{\frac{m-1}%
m}\leq& \int_M\{B_1|\nabla _{}\phi |^2v^2+B_3\phi ^2v^3\}dV_g \\
\leq& B_1\int_M|\nabla _{}\phi |^2v^2dV_g+B_3\left( \int_Mv^mdV_g\right)
^{\frac 1m}\left( \int_M(\phi v)^{\frac{2m}{m-1}}dV_g\right) ^{\frac{m-1}m}.
\end{align*}
Consequently
\begin{equation*}
\{\Lambda (M,g)-B_3\left( \int_Mv^mdV_g\right) ^{\frac 1m}\}\left(
\int_M(\phi v)^{\frac{2m}{m-1}}dV_g\right) ^{\frac{m-1}m}\leq
B_1\int_M|\nabla _{}\phi |^2v^2dV_g.
\end{equation*}
Under the assumption (4.11), we may choose sufficiently small $\varepsilon
_1 $ and $\varepsilon _2$ such that
\begin{equation*}
\Lambda (M,g)-B_3\left( \int_Mv^mdV_g\right) ^{\frac 1m}>0.
\end{equation*}
The remaining discussion is similar to that for Theorem 3.2.
\end{proof}

Now let us look at the case that $R>0$. By the solution of Yamabe problem,
we know that the Yamabe constant $\Lambda (M,g)$ is attained by a positive
function $u\in C^\infty (M)$. The metric $\widetilde{g}=u^{\frac 4{n-2}}g$ ($%
n=2m$), called the Yamabe metric, has constant scalar curvature given by
(cf. [He1], [LP]):
\begin{equation}
S_{\widetilde{g}}=\frac{2(2m-1)}{m-1}\Lambda (M,g)Vol(\widetilde{g})^{-\frac
1m}.  \tag{4.13}
\end{equation}
It is known that any Einstein metric on a compact Riemannian $n$-manifold
must be the Yamabe metric, provided it is not conformal to the standard
metric of $n$-sphere ([Ob]). Since $(M,g,J)$ is K\"ahler-Einstein, $g$ is
the Yamabe metric in its conformal class $[g]$. Hence (4.13) implies
\begin{equation}
\Lambda (M,g)=\frac{(m-1)R}{2m-1}Vol(M)^{\frac 1m}.  \tag{4.14}
\end{equation}
As in \S 3, we give two types of rigidity results for this case. The first
one is the following $L^m$-pinching result:

\begin{theorem}
 Let $(M,g,J)$ be a compact K\"ahler-Einstein $m$%
-manifold with $m\geq 2$ and $R>0$. Set
\begin{equation*}
Q(m)=\left\{
\begin{array}{cc}
\frac{m(m+3)}{m+1}, & m=2,3 \\
\frac{2(2m-1)}{m-1}, & m\geq 4.
\end{array}
\right.
\end{equation*}
If
\begin{equation}
||B||_{L^m(M)}<\Lambda (M,g)Q(m)\frac{\sqrt{m^2-1}}{3(m^2-2)},  \tag{4.15}
\end{equation}
then $(M,g,J)$ is biholomorphically homothetic to the complex projective
space $CP^m$.
\end{theorem}
\begin{proof}
By integrating (4.8) and using the H\"older inequality, we have
\begin{align*}
\frac{m+3}{m+1}\int_M|\nabla _{}|B||^2dV_g
\leq& \frac{3(m^2-2)}{\sqrt{m^2(m^2-1)}}||B||_{L^m(M)}\left( \int_M|B|^{\frac{%
2m}{m-1}}dV_g\right) ^{\frac{m-1}m}\\
&-\frac{2R}m\int_M|B|^2dV_g.  \tag{4.16}
\end{align*}
Applying (2.2) to $|B|$ leads to
\begin{equation}
\Lambda (M,g)\left( \int_M|B|^{\frac{2m}{m-1}}dV_g\right) ^{\frac{m-1}m}\leq
\int_M|\nabla _{}|B||^2dV_g+\frac{(m-1)R}{2m-1}\int_M|B|^2dV_g.  \tag{4.17}
\end{equation}
It follows from (4.16) and (4.17) that
\begin{align*}
\int_M|\nabla _{}|B||^2dV_g\leq& \frac{3(m+1)(m^2-2)||B||_{L^m(M)}}{\Lambda
(M,g)(m+3)\sqrt{m^2(m^2-1)}}\int_M\{|\nabla _{}|B||^2+\frac{(m-1)R}{2m-1}%
|B|^2\}dV_g \\
&-\frac{2(m+1)R}{m(m+3)}\int_M|B|^2dV_g.
\end{align*}
Consequently
\begin{align*}
\{1-\frac{3(m+1)(m^2-2)}{\Lambda (M,g)m(m+3)\sqrt{m^2-1}}||B||_{L^m(M)}\}%
\int_M|\nabla |B||^2dV_g &\\
+\{\frac{2(m+1)}{m(m+3)}-\frac{3(m^2-2)\sqrt{m^2-1}}{\Lambda
(M,g)m(m+3)(2m-1)}||B||_{L^m(M)}\}R\int_M|B|^2dV_g & \leq 0.  \tag{4.18}
\end{align*}
It is easy to verify that (4.15) implies that the two terms on the left hand
side of (4.18) are nonnegative. This leads to $B=0$, that is, $(M,g,J)$ has
constant holomorphic sectional curvature $\frac{2R}{m(m+1)}>0$. Then Synge's
theorem ensures that $M$ is simply connected. Hence $(M,g,J)$ is
biholomorphically homothetic to the complex projective space $CP^m$ with the
Fubini-Study metric (cf. Theorem 7.9 in [KN], Vol.II).
\end{proof}

\begin{remark} In [IK], Itho and Kobayashi gave a similar $L^m$-pinching
result to characterize $CP^m$. However, their pinching constant is an
abstract number depending on $n$ and $R$. Our pinching constant seems better and more explicit than theirs.
\end{remark}
The next one is the following pointwise pinching result.
\begin{theorem}
 Let $(M,g,J)$ be a compact K\"ahler-Einstein $m$%
-manifold with $m\geq 2$ and $R>0$. If
\begin{equation}
|B|<\frac{2\sqrt{m^2-1}R}{3(m^2-2)},  \tag{4.19}
\end{equation}
then $(M,g,J)$ is biholomorphically homothetic to the complex projective
space $CP^m$.
\end{theorem}
\begin{proof}
 From (4.6), we have
\begin{align*}
\frac 12\triangle |B|^2 \geq& |\nabla _{}B|^2+\left( \frac{2R}m-\frac{3(m^2-2)%
}{m\sqrt{m^2-1}}|B|\right) |B|^2 \\
\geq& \left( \frac{2R}m-\frac{3(m^2-2)}{m\sqrt{m^2-1}}|B|\right) |B|^2.  \tag{4.20}
\end{align*}
Under the assumption (4.19), the integration of (4.20) implies immediately
that $B=0$. Hence $(M,g,J)$ is biholomorphically homothetic to $CP^m$.
\end{proof}

\begin{remark}
 It is obvious that if the condition (4.19) is replaced by
\begin{equation}
|B|\leq \frac{2\sqrt{m^2-1}R}{3(m^2-2)},  \tag{4.21}
\end{equation}
then eihter $B=0$ or $|B|=\frac{2\sqrt{m^2-1}R}{3(m^2-2)}$ and $\nabla
_{}B=0 $. It would be interesting to investigate the case when the
equality of (4.21) holds.
\end{remark}
By combining Theorem 3.1 and Theorem 4.1, we get

\begin{theorem}
 Let $(M,g,J)$ be a complete noncompact K\"ahler $m$%
-manifold ($m\geq 2$) with zero scalar curvature and positive Yamabe
constant. If
\begin{equation*}
\sqrt{2}||E||_{L^m(M)}+||B||_{L^m(M)}<\frac{4\Lambda (M,g)(m^2+1)\sqrt{m^2-1}%
}{3m(m+1)(m^2-2)},
\end{equation*}
then $(M,g,J)$ has constant holomorphic sectional curvature $0$.
Furthermore, if $M$ is simply connected, then $(M,g,J)$ is biholomorphically
isometric to $C^m$.
\end{theorem}
\begin{proof}
 It is easy to verify that
\begin{equation*}
\frac{4\Lambda (M,g)(m^2+1)\sqrt{m^2-1}}{3m(m+1)(m^2-2)}<\frac{4\Lambda
(M,g)(m^2-m+1)}{m^3}\sqrt{\frac{2(m+1)(m+2)}{2m^2+4m+3}}
\end{equation*}
for $m\geq 2$. Then, by using Theorems 3.1, 4.1 successively, we may prove
the assertions.
\end{proof}

Since
\begin{equation*}
\frac{(m+3)\sqrt{m^2(m^2-1)}}{3(m+1)(m^2-2)}<\frac{(m+1)}m\sqrt{\frac{%
2(m+1)(m+2)}{2m^2+4m+3}},
\end{equation*}
Theorems 3.2, 4.2 imply that

\begin{theorem}
 Let $(M,g,J)$ be a complete noncompact K\"ahler $m$%
-manifold ($m\geq 4$) with constant negative scalar curvature and positive Yamabe
constant. Suppose
\begin{equation*}
\int_{B_r}(|E|^2+|B|^2)dV_g=o(r^2)\quad \text{as }r\rightarrow \infty \text{.%
}
\end{equation*}
If
\begin{equation*}
\sqrt{2}||E||_{L^m(M)}+||B||_{L^m(M)}<\frac{\Lambda (M,g)(m+3)\sqrt{%
m^2(m^2-1)}}{3(m+1)(m^2-2)},
\end{equation*}
then $(M,g,J)$ has constant holomorphic sectional curvature $\frac{2R}{m(m+1)%
}$.
\end{theorem}

One may verify that the pinching constant in Theorem 4.3 is smaller than
that in Theorem 3.3. Likewise, we have

\begin{theorem}
Let $(M,g,J)$ be a compact K\"ahler $m$-manifold ($%
m\geq 2$) with constant positive scalar curvature. Let $Q(m)$ be as in
Theorem 4.3. If
\begin{equation*}
\sqrt{2}||E||_{L^m(M)}+||B||_{L^m(M)}<\frac{\Lambda (M,g)Q(m)\sqrt{m^2-1}}{%
3(m^2-2)},
\end{equation*}
then $(M,g,J)$ is biholomorphically homothetic to $CP^m$.
\end{theorem}
Finally, Theorem 3.4 and Theorem 4.4 lead to

\begin{theorem} Let $(M,g,J)$ be a compact K\"ahler $m$-manifold ($%
m\geq 2$) with constant positive scalar curvature $R$. If
\begin{equation*}
\sqrt{2}|E|+|B|<\frac{2\sqrt{m^2-1}R}{3(m^2-2)},
\end{equation*}
then $(M,g,J)$ is biholomorphically homothetic to $CP^m$.
\end{theorem}

\noindent Tian Chong\\
School of Mathematical Science, Fudan University,
Shanghai 200433, China.\\
E-mail: valery4619@sina.com\\

\noindent Yuxin Dong\\
 School of Mathematical Science, Fudan University, Shanghai, 200433,  China.\\
E-mail: yxdong@fudan.edu.cn\\

\noindent Hezi Lin\\
 School of Mathematics and Computer Science,
 Fujian Normal University, Fuzhou,  350108,
 China.\\
E-mail: lhz1@fjnu.edu.cn\\

\noindent Yibin Ren\\
School of Mathematical Science, Fudan University,
Shanghai, 200433, China.\\
E-mail: allenrybqqm@hotmail.com


\begin{thebibliography}{99}

\bibitem[BKN]{BKN}  S. Bando, A. Kasue, H. Nakajima, On a construction of
coordinates at infinity manifolds with fast curvature decay and maximal
volume growth, Invent. Math. 97(1989) 313-349.

\bibitem[BM]{BM}  S. Bando, T. Mabuchi, Uniqueness of Einstein K\"ahler
metrics modulo connected group actions, Algebraic Geometry, Adv. Studies in
Pure Math., 10 (1987).

\bibitem[Bo]{Bo}  S. Bochner, Curvature Betti numbers, II, Ann. of Math. 50
(1949), 77-93.

\bibitem[Ch]{Ch}  X. Chen, Recent progress in K\"ahler geometry, ICM 2002,
Vol. III, 273-282.

\bibitem[Do]{Do}  S. Donaldson, Conjectures in Kahler geometry, Clay
Mathematics Proceedings, Vol. 3 (2004), 71-78.

\bibitem[DLR]{DLR}  Y. Dong, H. Lin and Y. Ren, Rigidity theorems for
complete Sasakian manifolds with constant scalar pseudo-Hermitian scalar
curvature,  arXiv:1402.6883 [math.DG].

\bibitem[He1]{He1}  E. Hebey, Variational methods and elliptic equations in
Riemannian geometry. Workshop on Recent Trends in Nonlinear Variational
Problems. Notes from lectures at ICTP, 2003.

\bibitem[He2]{He2}  E. Hebey, Nonlinear analysis on manifolds: Sobolev
Spaces and inequalities, American Mathematical Society, Courant Lecture
Notes in Mathematics, Vol. 5 (2000).

\bibitem[Hu]{Hu}  G. Huisken, Ricci deformation of the metric on a
Riemannian manifold, J. Diff. Geom. 21(1985), 47-62.

\bibitem[HV]{HV}  E. Hebey, M. Vaugon, Effective $L^p$ pinching for the
concircular curvature, J. Geom. Anal. 6 (1996), 531-553.

\bibitem[IN]{IN}  M. Itoh, T. Nakagawa, Variational stability and local
rigidity of Einstein metrics, Yokohama Mathematical Journal, Vol. 51, 2005.

\bibitem[IS]{IS}  M. Itoh, H. Satoh, Isolation of the Weyl conformal tensor
for Einstein manifolds, Proc. Japan Acad., 78, Ser. A (2002), 140-142

\bibitem[IK]{IK}  M. Itoh, D. Kobayashi, Isolation theorems of the Bochner
curvature type tensors, Tokyo J. Math. 27 (2004), 227-237.

\bibitem[Ki]{Ki}  S. Kim, Rigidity of noncompact complete manifolds with
harmonic curvature, Manuscripta Math. 135 (2011), 107-116.

\bibitem[KN]{KN}  S. Kobayashi, K. Nomizu, Foundations of differential
geometry, Vol. II, Wiley Classics Library Edition Published 1996.

\bibitem[Ku]{Ku}  M. Kuhnel, Complete K\"ahler-Einstein manifolds. In
Complex and Differential Geometry, Springer Berlin Heidelberg (2011),
171-181.

\bibitem[LP]{LP}  J. Lee, T. Parker, The Yamabe problem, Bull. A. M. S. 17
(1987), 37-91.

\bibitem[LS]{LS} C. LeBrun, S.R. Simanca, Extremal K\"ahler metrics and complex deformation theory,
    Geom. Func. Analysis, Vol. 4, No. 3(1994), 298-336.


\bibitem[Ob]{Ob}  M. Obata, The conjectures on conformal transformations of
Riemannian manifolds, J. Diff. Geom. 6 (1971), 247-258.

\bibitem[Ok]{Ok}  M. Okumura. Hypersurfaces and a pinching problem on the
second fundamental tensor. Amer. J. Math., 96 (1974), 207-213.

\bibitem[PRS]{PRS}  S. Pigola, M. Rigoli and A. G. Setti, Some
characterizations of space-forms, Trans. Amer. Math. Soc. 359 (2007),
1817-1828.

\bibitem[Sh1]{Sh1}  Z. Shen, Some rigidity phenomena for Einstein metrics,
Proc. Amer. Math. Soc. 108 (1990), 981-987

\bibitem[Sh2]{Sh2}  Z. Shen, Rigidity theorems for nonpositive Einstein
metrics, Proc. Amer. Math.Soc. 116 (1992), 1107-1114

\bibitem[SY]{SY}  R. Schoen, S.T. Yau, Conformally flat manifolds, Kleinian
groups and scalar curvature, Invent. Math. 92 (1988), 47-71.

\bibitem[Ti]{Ti}  G. Tian, Extremal metrics and geometric stability, Houston
J. of Math., Vol. 28 (2002), 411-432.

\bibitem[TY1]{TY1}  G. Tian, S.T. Yau, Complete Kahler manifolds with zero
Ricci curvature, I, J. Amer. Math. Soc. 3 (1990), 579-610.

\bibitem[TY2]{TY2}  G. Tian, S.T. Yau, Complete Kahler manifolds with zero
Ricci curvature, II, Invent. Math. 106 (1991), 27-60.

\end{thebibliography}
\end{document}